\documentclass[preprint]{imsart}

\usepackage[english]{babel}
\usepackage[T1]{fontenc}
\usepackage[utf8]{inputenc}
\usepackage{lmodern}

\usepackage{calc}
\usepackage{amsmath}
\usepackage{amsfonts}
\usepackage{amssymb}
\usepackage{amsthm}
\usepackage{euscript}
\usepackage[numbers]{natbib}
\usepackage{xspace}
\usepackage{bm}
\usepackage{enumerate}
\usepackage{hyperref}
\usepackage{graphicx}
\usepackage{subfig}
\usepackage{mathabx}

\startlocaldefs

\newtheorem{lemma}{Lemma}[section]
\newtheorem{theorem}{Theorem}
\newtheorem{proposition}[theorem]{Proposition}

\theoremstyle{definition}

\newtheorem{remark}{Remark}

\numberwithin{equation}{section}

\DeclareMathOperator{\diam}{diam}

\newcommand{\R}{\ensuremath{\mathbb{R}}\xspace}

\newcommand{\N}{\ensuremath{\mathbb{N}}\xspace}

\newcommand{\Pii}{\ensuremath{\mathcal P}\xspace}

\newcommand{\Si}{\ensuremath{\mathcal S}\xspace}

\newcommand{\dimH}{\ensuremath{{\dim}_{\text{\normalfont\tiny H}}}\xspace}

\newcommand{\pth}[1]{(#1)}
\newcommand{\pthb}[1]{\bigl(#1\bigr)}
\newcommand{\pthB}[1]{\Bigl(#1\Bigr)}
\newcommand{\pthbb}[1]{\biggl(#1\biggr)}

\newcommand{\bkt}[1]{[#1]}
\newcommand{\bktb}[1]{\bigl[#1\bigr]}
\newcommand{\bktB}[1]{\Bigl[#1\Bigr]}
\newcommand{\bktbb}[1]{\biggl[#1\biggr]}

\newcommand{\brc}[1]{\{#1\}}
\newcommand{\brcb}[1]{\bigl\{#1\bigr\}}
\newcommand{\brcB}[1]{\Bigl\{#1\Bigr\}}
\newcommand{\brcbb}[1]{\biggl\{#1\biggr\}}

\newcommand{\scpr}[2]{\langle #1,#2 \rangle}

\newcommand{\scprbb}[2]{\biggl\langle #1,#2 \biggr\rangle}

\newcommand{\dt}{\ensuremath{\mathrm d}\xspace} 
\newcommand{\eqdef}{:=}

\newcommand{\ivoo}[1]{\ensuremath{(#1)}}

\newcommand{\ivfo}[1]{\ensuremath{[#1)}}

\newcommand{\ivfobb}[1]{\ensuremath{\biggl[#1\biggr)}}

\newcommand{\ivff}[1]{\ensuremath{[#1]}}

\newcommand{\abs}[1]{\lvert#1\rvert}
\newcommand{\absb}[1]{\bigl\lvert#1\bigr\rvert}
\newcommand{\absB}[1]{\Bigl\lvert#1\Bigr\rvert}

\newcommand{\norm}[1]{\lVert#1\rVert}
\newcommand{\normb}[1]{\bigl\lVert#1\bigr\rVert}

\newcommand{\Esp}{\ensuremath{\mathbb E}\xspace}

\newcommand{\prb}[2][]{\mathbb{P}#1\pthb{\hspace{1pt}#2\hspace{1pt}}}
\newcommand{\prB}[2][]{\mathbb{P}#1\pthB{#2}}

\newcommand{\prcb}[3][]{\mathbb{P}#1\pthb{\hspace{1pt}#2\bigm|#3\hspace{1pt}}}

\newcommand{\esp}[2][]{\mathbb{E}#1\bkt{#2}}
\newcommand{\espb}[2][]{\mathbb{E}#1\bktb{\hspace{1pt}#2\hspace{1pt}}}

\newcommand{\espbb}[2][]{\mathbb{E}#1\bktbb{#2}}

\newcommand{\var}[2][]{\mathrm{Var}#1\pth{#2}}

\newcommand{\varbb}[2][]{\mathrm{Var}#1\pthbb{#2}}

\newcommand{\varcb}[3][]{\mathrm{Var}#1\pthb{\hspace{1pt}#2\bigm|#3\hspace{1pt}}}

\newcommand{\indi}{\ensuremath{\mathbf{1}}\xspace}
\newcommand{\eps}{\varepsilon}

\newcommand{\vsp}{\vspace{.15cm}}

\endlocaldefs

\begin{document}

\begin{frontmatter}

\title{Image sets of fractional Brownian sheets}
\runtitle{Image sets of fractional Brownian sheets}

\author{\fnms{Paul} \snm{Balan\c{c}a}%
\thanksref{t2}%
\ead[label=e1]{paul.balanca@gmail.com}%
\ead[label=u1,url]{balancap.github.io}%
}%
\thankstext{t2}{Research partially supported by the French Embassy in Israel.}%
\address{
Faculty of Industrial Engineering and Management\\%
Technion Israel Institute of Technology\\%
Haifa 32000, Israël\\[1ex]%
\printead{e1}\\\printead{u1}%
}

\affiliation{Technion}
\runauthor{Paul Balan\c{c}a}

\begin{abstract}
  Let \(B^H = \brc{ B^H(t), t\in\R^N }\) be an $(N,d)$-fractional Brownian sheet with Hurst index \(H=(H_1,\dotsc,H_N)\in (0,1)^N\). The main objective of the present paper is to study the Hausdorff dimension of the image sets \(B^H(F+t)\), $F\subset\R^N$ and $t\in\R^N$, in the dimension case \(d<\tfrac{1}{H_1}+\cdots+\tfrac{1}{H_N}\). Following the seminal work of \citet{Kaufman-1989}, we establish uniform dimensional properties on \(B^H\), answering questions raised by \citet{Khoshnevisan.Wu.ea-2006} and \citet{Wu.Xiao-2009}.

  For the purpose of this work, we introduce a refinement of the \emph{sectorial local-nondeterminism} property which can be of independent interest to the study of other fine properties of fractional Brownian sheets.
\end{abstract}



\end{frontmatter}


\section{Introduction}

In the last thirty years, several extensions of the well-known fractional Brownian motion (fBm) introduced by \citet{Mandelbrot.VanNess-1968} have emerged in the Gaussian random fields literature. Two major classes of multiparameter processes have been defined: Lévy's $N$-parameter (fractional) Brownian motion and fractional Brownian sheets. The first one is an isotropic process known to be be locally non-deterministic (LND, see \citet{Pitt-1978} for a more complete reference), self-similar and with stationary increments. As a consequence, the geometry and fine properties of the $N$-parameter fractional Brownian motion have been extensively documented by extending the classic techniques developed in the literature related to the fractal geometry of the one-dimensional fractional Brownian motion.\vsp

On the another hand, the understanding of fractional Brownian sheets (fBs) introduced by \citet{Kamont-1996} has been proved to be more challenging and technical as the former does not satisfy the classic LND property. Recall that an $(N,d)$-fractional Brownian sheet $B^H = \brc{B^H(t), t\in\R^N}$ with Hurst index $H=(H_1,\dotsc,H_N)\in\ivoo{0,1}^N$ is defined as a centered Gaussian process with independent and identically distributed components whose covariance is given by
\begin{equation}  \label{eq:fbs_cov}
  \espb{ B_0^H(s) B_0^H(t) } = \prod_{\ell=1}^N \bktB{ \abs{s_\ell}^{2H_\ell} + \abs{t_\ell}^{2H_\ell} -\abs{s_\ell-t_\ell}^{2H_\ell} }\quad s,t\in\R^N.
\end{equation}
Note that similarly to fBm, it also admits an integral representation with respect to the Brownian sheet $W$:
\begin{equation}  \label{eq:fbs_int_rep}
  B_0^H(t) = \int_{R^N} \prod_{\ell=1}^N \brcB{ (t_\ell-u_\ell)_+^{H_\ell-1/2} - (-u_\ell)_+^{H_\ell-1/2} } \,\dt W_u.
\end{equation}
In the case $H_1=\cdots=H_N=\tfrac{1}{2}$, we obtain the well-known Brownian sheet.

Anisotropic Gaussian random fields such as fBs have raised an increasing interest in recent years as they appear naturally in the study of stochastic partial differential equations (SPDEs) and Markov processes \cite{Mueller.Tribe-2002}. In a more applied perspective, several phenomena in image processing, hydrology and spatial statistics \cite{Davies.Hall-1999,Benson.Meerschaert.ea-2006} are intrinsically anisotropic, and thus, required the introduction of such theoretical models.

The study of distributional properties of fractional Brownian sheets have been considerably eased with the introduction by \citet{Khoshnevisan.Xiao-2007} of the so-called \emph{sectorial local-nondeterminism} property. Namely, the latter states that for any $u,v,t^1,\dotsc,t^n\in\ivfo{\eps,+\infty}^N$,
\begin{equation}   \label{eq:sec_local_nondet1}
  \varcb{ B^H(u) }{ B^H(t^1),\dotsc,B^H(t^n) } \geq c_0 \sum_{\ell=1}^N \min_{1\leq j\leq n} \absb{ u_\ell - t^j_\ell }^{2 H_\ell}
\end{equation}
where the constant $c_0>0$ only depends on $\eps$. Firstly introduced on the Brownian sheet, it has been then extended to general fractional Brownian sheets by \citet{Wu.Xiao-2007}.

This sectoral LND has been the cornerstone to the study of multiple geometrical properties of (fractional) Brownian sheets, allowing to adapt classic techniques used on multiparameter fractional Brownian motion to this class of processes. More precisely, the distributional properties of the local time and level sets have been investigated by \citet{Khoshnevisan.Xiao-2007,Ayache.Wu.ea-2008}, extending earlier works by \citet{Dalang.Walsh-1993,Xiao.Zhang-2002}. The fractal geometry of image sets $B^H(F)$ has also been extensively studied in \cite{Khoshnevisan.Wu.ea-2006,Khoshnevisan.Xiao-2007,Wu.Xiao-2007a,Wu.Xiao-2007}. Note that as pointed out by \citet{Xiao-2009a}, due to the anisotropic nature of fractional Brownian sheets, it is usually convenient to study geometrical properties using the following anisotropic metric $\rho$:
\begin{equation}  \label{eq:anisotropic_metric}
  \forall s,t\in\R^N;\quad \rho(s,t) = \sum_{\ell=1}^N \abs{s_\ell - t_\ell}^{H_\ell}.
\end{equation}

As previously outlined, we aim in this work to investigate uniform dimensional properties of image sets $B^H(F)$ of fractional Brownian sheets. The high dimension case $d\geq\tfrac{1}{H_1}+\cdots+\tfrac{1}{H_N}$ has been thoroughly discussed by \citet{Khoshnevisan.Wu.ea-2006,Wu.Xiao-2007a} who obtained the following result: with probability one,
\begin{equation}  \label{eq:unif_fbs}
  \text{for every Borel set } F\subset\R^N;\quad \dimH B^H(F) = \dimH^\rho F,
\end{equation}
where $\dimH^\rho$ designates the Hausdorff dimension with respect to the anisotropic metric $\rho$. Note that a similar result exists on the multiparameter fractional Brownian motion when $d\geq \tfrac{N}{\alpha}$ (see the work of \citet{Monrad.Pitt-1987}).\vsp

When $d<\tfrac{1}{H_1}+\cdots+\tfrac{1}{H_N}$, the previous uniform result does not hold any more. For instance, it is obviously false if one considers the level set $F=W^{-1}(0)$, where $W$ is a Brownian sheet: \citet{Khoshnevisan.Xiao-2007} have proved that $F$ has positive Hausdorff dimension when $d < 2N$, whereas we clearly have $\dimH W(F) = 0$. Nevertheless, following the ideas developed by \citet{Kaufman-1989}, one may hope to establish a weaker uniform Hausdorff dimension result. Indeed, the former has proved that a one-dimensional Brownian motion satisfies a slightly weaker property: with probability one, for every Borel set $F\subset\R$,
\begin{equation}  \label{eq:unif_fbs_aa}
  \dimH B(F+t) = 2\dimH F\quad\text{for almost all }t\in\R.
\end{equation}

\citet{Khoshnevisan.Wu.ea-2006,Wu.Xiao-2007a} have investigated the extension of this property to the $(N,1)$ Brownian sheet and fractional Brownian sheets satisfying $H_N d \leq 1$ (assuming that $H_1\leq\cdots\leq H_N$). Even though it may seem to natural that the former result would hold for any fBs such that $d<\tfrac{1}{H_1}+\cdots+\tfrac{1}{H_N}$, this question was left opened in the previous works as the authors observed that techniques based on sectorial LND do not seem to scale well the general case (on contrary to the $(N,d)$-fractional Brownian motion considered by \citet{Wu.Xiao-2007}).\vsp

Consequently, the main purpose of this work is to close the gap between the statement of \citet{Wu.Xiao-2007a} and the uniform case presented in Equation~\eqref{eq:unif_fbs}. More precisely, we prove in Section~\ref{sec:results} the following two uniform results on the geometry of fractional Brownian sheets.
\begin{theorem}  \label{th:weak_unif_dim_fbs}
  Let $B^H$ be a fractional Brownian sheet and suppose $d<\sum_{\ell=1}^N \frac{1}{H_\ell}$. Then, with probability $1$, for every Borel set $F\subseteq\ivoo{0,\infty}^N$,
  \begin{equation}
    \dimH B^H(F+t) = \min\brcb{ d, \dimH^\rho F }  \quad \text{for almost all }t\in\R_+^N.
  \end{equation}
\end{theorem}
In addition, we also extend the result of \citet{Wu.Xiao-2007a} related to Lebesgue measure of image sets.
\begin{theorem}  \label{th:lebesgue_fbs}
  Let $B^H$ be a fractional Brownian sheet and suppose $d<\sum_{\ell=1}^N \frac{1}{H_\ell} $. Then, with probability $1$, for every Borel set $F\subset\ivoo{0,\infty}^N$ such that $\dimH^\rho F > d$,
  \[
    \lambda_d(B^H(F+t))>0  \quad\text{for almost all } t\in\R_+^N,
  \]
  where $\lambda_d$ denotes the Lebesgue measure on $\R^d$.
\end{theorem}
The proof of the two previous results rely on the introduction in Proposition~\ref{prop:ani_lnd_fbs} (Section~\ref{sec:lnd}) of an anisotropic local non-determinism property different from sectorial LND. The former then allows to adapt the seminal methods of \citet{Kaufman-1989} to fractional Brownian sheets satisfying $d<\sum_{\ell=1}^N \frac{1}{H_\ell}$. Note that we hope that this anisotropic LND property can also be of independent interest to the study of remaining open questions on the fractal geometry of fractional Brownian sheets and more general anisotropic Gaussian random fields.

\section{Anisotropic local nondeterminism}  \label{sec:lnd}

The local nondeterminism property has historically been introduced by \citet{Berman-1973} in the study of local times of Gaussian processes. Since then, it has been widely and successfully used to obtain multiple fine sample paths properties of Gaussian processes, including small balls probabilities, level sets and Hausdorff dimension of graphs and image sets. We refer to the surveys of \citet{Geman.Horowitz-1980,Xiao-2006} for a more precise overview on the subject.

As previously outlined, the Brownian sheet, and thereby fractional Brownian sheets, are known to be non locally non-deterministic and \citet{Khoshnevisan.Xiao-2007} have introduced the \emph{sectorial local nondeterminism} property in order to still being able to investigate distributional properties of this class of processes. The simplest form has been presented in the introduction, Equation~\eqref{eq:sec_local_nondet1}. In order to investigate uniform dimension of image sets, one needs an analogue of the former on increments. Namely, \citet{Wu.Xiao-2007a} have proved that any fractional Brownian sheet satisfies for every $s,t,s^1,\dotsc,s^n\in\ivfo{\eps,+\infty}$,
\begin{align}  \label{eq:sec_local_nondet2}
  &\varcb{ B^H(s) - B^H(t) }{ B^H(s^1),\dotsc,B^H(s^n) } \nonumber \\
  &\geq c_1 \sum_{\ell=1}^N \min\brcbb{ \min_{1\leq j\leq n} \absb{ s_\ell - s^j_\ell }^{2 H_\ell} + \min_{1\leq j\leq n} \absb{ t_\ell - s^j_\ell }^{2 H_\ell}, \abs{s_\ell-t_\ell}^{2 H_\ell} },
\end{align}
where the constant $c_1>0$ only depends on $\eps$.

Nevertheless, it appears in the work of \citet{Khoshnevisan.Wu.ea-2006} that the previous \emph{sectorial local nondeterminism} property is not sufficiently fine to extend the result~\eqref{eq:unif_fbs_aa} to any fractional Brownian sheet satisfying $d<\sum_{\ell=1}^N \frac{1}{H_\ell}$. Consequently, we present in the following proposition a refinement of the former.
\begin{proposition}  \label{prop:ani_lnd_fbs}
  Suppose $B^H$ is a fractional Brownian sheet and $\eps>0$. Then, there exists a constant $c_{0}>0$ such that for every $t,s,s^1,\dotsc,s^n\in\ivfo{\eps,1}^N$.
  \begin{align}  \label{eq:ani_lnd_fbs}
    \varcb{ B^H(t) - B^H(s) }{ B^H(s^1),\dotsc,B^H(s^n) }
    \geq c_0\, \rho(s,t)^2 \cdot \sum_{\ell=1}^N r_\ell^{2H_\ell},
  \end{align}
  where for any $\ell\in\brc{1,\dotsc,N}$, we define
  \[
    r_\ell \eqdef \min_{1\leq j\leq n} \abs{ s_\ell - s^j_\ell } + \min_{1\leq j\leq n} \abs{ t_\ell - s^j_\ell }.
  \]
\end{proposition}
\begin{proof}
  We aim to prove a property slightly stronger than Equation~\eqref{eq:ani_lnd_fbs}. Namely, for every $t,s,s^1,\dotsc,s^n\in\ivfo{\eps,1}^N$,
  \begin{align}  \label{eq:ani_lnd_fbs2}
    &\varcb{ B^H(t) - B^H(s) }{ B^H(s^1),\dotsc,B^H(s^n) } \nonumber \\
    &\geq c_0 \sum_{\ell=1}^N \min\brcb{ r_\ell^{2H_\ell}, \abs{s_\ell - t_\ell}^{2H_\ell} } +
    c_0 \sum_{\ell=1}^N r_\ell^{2H_\ell} \cdot \brcbb{ \sum_{i\neq \ell} \abs{ s_i - t_i }^{2H_i} }.
  \end{align}
  We easily observe that since $\min\brcb{ r_k^{2H_k}, \abs{s_k - t_k}^{2H_k} }\geq r_k^{2H_k}\cdot\abs{s_k - t_k}^{2H_k}$, the former clearly induces Inequality~\eqref{eq:ani_lnd_fbs}. In addition, in the two components appearing in Equation~\eqref{eq:ani_lnd_fbs2}, the first one is a clear consequence of the sectorial LND property \eqref{eq:sec_local_nondet2}, since $r_\ell = \min_{j\leq n} \abs{ s_\ell - s^j_\ell } + \min_{j\leq n} \abs{ t_\ell - s^j_\ell }$.

  Hence, we may focus on the second part, and, set $k\in\brc{1,\dotsc,N}$ and $i\neq k$. Without any loss of generality, we may assume that $t_i - s_i\geq 0$ (unless, simply permute $t$ and $s$).
  If $\abs{s_k-t_k} \geq r_k$, we simply observe that
  \begin{align*}
    \min\brcb{ r_k^{2H_k}, \abs{s_k - t_k}^{2H_k} } = r_k^{2H_k} \geq r_k^{2H_k}\cdot \abs{s_i - t_i}^{2H_i}.
  \end{align*}
  The combination of the previous remark and the sectorial LND property yield the expected inequality. Therefore, we may assume in the sequel that $\abs{s_k-t_k} < r_k$.

  Owing to the integral representation \eqref{eq:fbs_int_rep} of fractional Brownian sheets, we know that
  \begin{align*}
    &\varcb{ B^H(t)-B^H(s) }{ B^H(s^1),\dotsc,B^H(s^n) } \\
    &= \inf_{\alpha\in\R^n} \espbb{ \pthB{ B^H(t)-B^H(s) - \sum_{j=1}^n \alpha_j B^H(s^j) }^2
     } \\
    &= \inf_{\alpha\in\R^n} \int_{\R^N} \pthbb{ K(u,t,H) - K(u,s,H) - \sum_{j=1}^n \alpha_j K(u,s^j,H) }^2 \dt u.
  \end{align*}
  where $K(u,t,H) \eqdef \prod_{\ell=1}^N \brcb{ (t_\ell-u_\ell)_+^{H_\ell-1/2} - (-u_\ell)_+^{H_\ell-1/2}}$. The previous expression of the conditional variance can be lower bounded by
  \begin{align*}
    \inf_{\alpha\in\R^n} \int_{\R_+^N} \pthbb{ \prod_{\ell=1}^N (t_\ell-u_\ell)_+^{H_\ell-1/2} - \prod_{\ell=1}^N (s_\ell-u_\ell)_+^{H_\ell-1/2} - \sum_{j=1}^n \alpha_j \prod_{\ell=1}^N (s^j_\ell-u_\ell)_+^{H_\ell-1/2} }^2 \dt u.
  \end{align*}
  In order to obtain a uniform lower bound of the previous expression, the main idea is to exhibit an element $h\in L^2$ which is orthogonal to the family of functions $\prod_{\ell=1}^N (s^j_\ell-u_\ell)_+^{H_\ell-1/2}$. Hence, let us define
  \begin{align*}
    h(u) = h_k(u_k) \cdot
    \indi_{\ivff{s_i,t_i}}(u_i) \cdot \prod_{\ell\neq i,k} \indi_{\ivff{0,\eps}}(u_\ell).
  \end{align*}
  where
  $h_k(u_k) = (u_k-t_k+r_k)_+^{1/2-H_k} + (u_k-t_k-r_k)_+^{1/2-H_k} - 2(u_k-t_k)_+^{1/2-H_k}$. Note that the support of the function $h$ is included in the set $\ivfo{t_k-r_k,\infty}\times\ivff{s_i,t_i}\times\ivff{0,\eps}^{N-2}$ (up to a permutation of variables). In addition, when $H_k = 1/2$, $h_k$ simply corresponds to the difference $\indi_{\ivfo{t_k-r_k,t_k}}-\indi_{\ivfo{t_k,t_k+r_k}}$.
  Let us prove $h$ is well-designed for our purpose by evaluating the scalar product, for any fixed $j\in\brc{1,\dots,n}$:
  \begin{align*}
    \scprbb{ h }{ \prod_{\ell=1}^N (s^j_\ell-u_\ell)_+^{H_\ell-1/2} } =
    c \int_{\R_+} (s^j_k-u_k)_+^{H_k-1/2} \cdot h_k(u_k) \,\dt u_k,
  \end{align*}
  where $c$ corresponds to the integration over variables $u_\ell$, $\ell\neq k$. We need to distinguish two different cases, depending on the value of $s^j_k$.
  \begin{enumerate}[ \it 1.]
    \item If $s^j_k < t_k - r_k$, $(s^j_k-u_k)_+^{H_k-1/2}$ and $h_k$ have disjoint supports (respectively $\ivff{0,s^j_k}$ and $\ivfo{t_k-r_k,\infty}$), and therefore, the inner product is clearly equal to zero.\vsp

    \item If $s^j_k \geq t_k - r_k$,
    \begin{align*}
      \scprbb{ h }{ \prod_{\ell=1}^N (s^j_\ell-u_\ell)_+^{H_\ell-1/2} }
      = c\,\int_{\ivff{t_k-r_k,s^j_k}} (s^j_k-u_k)^{H_k-1/2} \cdot h_k(u_k) \,\dt u_k.
    \end{align*}
    Let $a$ denotes either $0$, $-r_k$ or $r_k$. Then,
    \begin{align*}
      &\int_{\ivff{t_k+a,s^j_k}} (s^j_k-u_k)^{H_k-1/2} \cdot (u_k - t_k - a)^{1/2-H_k} \,\dt u_k \\
      &= (s^j_k-t_k-a) \cdot \int_{\ivff{0,1}} v_k^{H_k-1/2} (1-v_k)^{1/2-H_k} \,\dt v_k,
    \end{align*}
    using the change of variable $v_k = (s^j_k-u_k) / (s^j_k-t_k-a)$. Consequently,
    \begin{align*}
      \scprbb{ h }{ \prod_{\ell=1}^N (s^j_\ell-u_\ell)_+^{H_\ell-1/2} }
      = c\brcB{ (s^j_k-t_k-r_k) + (s^j_k-t_k+r_k) - 2(s^j_k-t_k) } = 0.
    \end{align*}
  \end{enumerate}
  The function $h$ is orthogonal to any $\prod_{\ell=1}^N (s^j_\ell-u_\ell)_+^{H_\ell-1/2}$, and therefore to the linear space spanned by the previous collection. As a consequence,
  \begin{align*}
    &\varcb{ B^H(t)-B^H(s) }{ B^H(s^1),\dotsc,B^H(s^n) } \\
    &\geq \frac{ 1 }{ \norm{h}^2_{L^2} } \scprbb{ h }{ \prod_{\ell=1}^N (t_\ell-u_\ell)_+^{H_\ell-1/2} - \prod_{\ell=1}^N (s_\ell-u_\ell)_+^{H_\ell-1/2} }^2 \\
    &= \frac{ 1 }{ \norm{h}^2_{L^2} } \scprbb{ h }{ \prod_{\ell=1}^N (t_\ell-u_\ell)_+^{H_\ell-1/2} }^2,
  \end{align*}
  since the support of $h$ does not intersect $\ivff{0,s}$ due to the component $\indi_{\ivff{s_i,t_i}}(u_i)$ in the former. Let us first estimate the norm $\norm{h}^2_{L^2}$:
  \begin{align*}
    \norm{h}^2_{L^2}
    &= \eps^{N-2} \cdot \abs{s_i-t_i} \\
    &\cdot \int_{\R} \brcB{ (u_k-t_k+r_k)_+^{1/2-H_k} + (u_k-t_k-r_k)_+^{1/2-H_k} - 2(u_k-t_k)_+^{1/2-H_k} }^2 \dt u_k \\
    &= \eps^{N-2} \cdot \abs{s_i-t_i} \cdot r_k^{2-2H_k}\int_{\R_+} \brcB{ v_k^{1/2-H_k} + (v_k-2)_+^{1/2-H_k} - 2(v_k-1)_+^{1/2-H_k} }^2 \dt v_k\\
    &= c_0 \,\abs{s_i-t_i} \cdot r_k^{2-2H_k}.
  \end{align*}
  The previous integral is finite since $1-2H_k > -1$ and $v_k^{1/2-H_k} + (v_k-2)_+^{1/2-H_k} - 2(v_k-1)_+^{1/2-H_k}\sim_{\infty} v_k^{-1-2H_k}$. On the other hand, the inner product is equal to
  \begin{align*}
     \scprbb{ h }{ \prod_{\ell=1}^N (t_\ell-u_\ell)_+^{H_\ell-1/2} }
     &= \prod_{\ell\neq i,k} \int_{0}^{\eps} (t_\ell-u_\ell)^{H_i-1/2} \,\dt u_\ell
     \int_{s_i}^{t_i} (t_i-u_i)^{H_i-1/2} \,\dt u_i \\
     & \times \int_{t_k-r_k}^{t_k} (t_k-u_k)^{H_k-1/2} (u_k-t_k+r_k)^{1/2-H_k} \,\dt u_k \\
     &\geq c_2\, \abs{s_i-t_i}^{H_i+1/2} \cdot r_k.
  \end{align*}
  still using a similar change of variables and observing that $c_2>0$ only depends on $\eps$. Hence, we eventually obtain
  \begin{align*}
    \varcb{ B^H(t)-B^H(s) }{ B^H(s^1),\dotsc,B^H(s^n) } \geq c_3\, r_k^{2H_k} \cdot \abs{s_i-t_i}^{2H_i},
  \end{align*}
  where the constant $c_3>0$ only depends on $N$, $\eps$ and $H$. This last inequality then clearly leads to the second term in Equation~\eqref{eq:ani_lnd_fbs2}.
\end{proof}

\begin{remark}
  We may note that the local nondeterminism property presented in Proposition~\ref{prop:ani_lnd_fbs} is not sensu stricto an extension of the sectorial LND property~\eqref{eq:sec_local_nondet2}. Indeed, one can simply observe that if the two terms $\rho(s,t)^2$ and $\sum_{\ell=1}^N r_\ell^{2H_\ell}$ are of same order, then the sectorial LND bound \eqref{eq:sec_local_nondet2} is tighter.

  \begin{figure}[!ht]
    \includegraphics[scale=1.0]{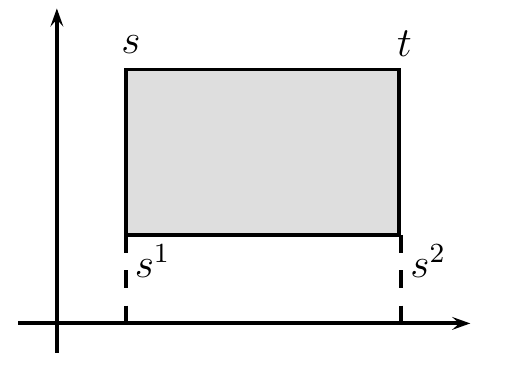}
    \caption{Example of conditional variance $\varcb{ B(t)-B(s) }{ B(s^1),B(s^2) }$}
    \label{fig:ex_rslnd_fbs}
  \end{figure}
  On the other hand, Figure~\ref{fig:ex_rslnd_fbs} illustrates the typical case where the anisotropic LND property~\eqref{eq:ani_lnd_fbs} provides a better estimate than the classic sectorial LND. Namely, if $W$ is a two dimensional Brownian sheet, we easily observe that the bound given by Equation~\eqref{eq:sec_local_nondet2} on $\varcb{ W(t)-W(s) }{ W(s^1),W(s^2) }$ is zero. On the other hand, Proposition~\ref{prop:ani_lnd_fbs} gives an optimal lower bound, proportional to the variance of the term $W(t)-W(s)-W(s^2)+W(s^1)$ (informally equal to $\var{W(\text{``grey area''})}$). This improvement corresponding to some specific geometrical configurations will be the cornerstone in the proofs of Theorems~\ref{th:weak_unif_dim_fbs} and \ref{th:lebesgue_fbs}.
\end{remark}

\begin{remark}
  The calculus presented in the proof of Proposition~\ref{prop:ani_lnd_fbs} offers an alternative way to prove the sectorial LND property. \citet{Wu.Xiao-2007a}, and originally \citet{Kahane-1985}, used estimates on the Fourier representation of fractional Brownian sheets to obtain the lower bound, whereas our proof is based on the classic time integral representation \eqref{eq:fbs_int_rep}.
\end{remark}

\section{Weak uniform Hausdorff dimension of image sets}  \label{sec:results}

Based on the refinement obtained in the previous section, we now extend the weak uniform Hausdorff results presented by \citet{Kaufman-1989}, \citet{Khoshnevisan.Wu.ea-2006} and \citet{Wu.Xiao-2007a}. The structure of the proof of Theorem~\ref{th:weak_unif_dim_fbs} follows the ideas initially described by \citet{Kaufman-1989} and relies mainly on the estimate obtained in the following lemma.
\begin{lemma}  \label{lemma:moments_IxyR}
  Define
  \begin{align*}
    I(x,y,R) = \int_{\ivff{\eps,1}^N} \indi_{\ivff{-1,1}}\pthb{ R\cdot\normb{ B^H(x+t) - B^H(y+t) } } \,\dt t,
  \end{align*}
  Then, for all $R>0$, $x,y\in\ivff{\eps,1}^N$ and integers $p\geq 1$,
  \begin{equation}  \label{eq:lemma_IxyR}
    \espb{ \pthb{ I(x,y,R) }^p } \leq c_0^p (p!)^N R^{-dp} \rho(x,y)^{-dp},
  \end{equation}
  where the constant $c_0$ only depends on $\eps$.
\end{lemma}
\begin{proof}
  Since $B_1^H,\dotsc,B^H_d$ are independent copies of an $(N,1)$-fractional Brownian sheet $B^H_0$, the $p$-th moment of $I(x,y,R)$ is equal to
  \begin{align*}
    \esp{ (I(x,y,R))^p }
    &= \int_{\ivff{\eps,1}^{Np}} \prb{ \norm{ B^H(x+t^j) - B^H(y+t^j) } \leq R^{-1} , 1\leq j \leq p } \,\dt t^1 \cdots \dt t^p \\
    &= \int_{\ivff{\eps,1}^{Np}} \prb{ \abs{ B_0^H(x+t^j) - B_0^H(y+t^j) } \leq R^{-1} , 1\leq j \leq p }^d \,\dt t^1 \cdots \dt t^p.
  \end{align*}
  We will bound the previous integral by induction on the parameter $p$. Hence, let us fix fix $t^1,\dotsc, t^{p-1} \in\ivff{\eps,1}^N$ and integrate over the variable $t^p$. Note that without any loss of generality, we assume that all coordinates of $t^1,\dotsc, t^{p-1}$ are distinct.

  The distribution of $B_0^H(x+t^p) - B_0^H(y+t^p)$ conditionally to $B_0^H(x+t^j) - B_0^H(y+t^p)$, $j\in\brc{1,\dotsc,p-1}$ is clearly centered and Gaussian. Therefore,
  \begin{align*}
    \Pii(t^p)
    &\eqdef\prcb{ \abs{ B_0^H(x+t^p) - B_0^H(y+t^p) } \leq R^{-1} }{ \abs{B_0^H(x+t^j) - B_0^H(y+t^j)}, 1\leq j\leq p-1 } \\
    &\leq R^{-1} \cdot \varcb{ B_0^H(x+t^p) - B_0^H(y+t^p) }{ B_0^H(x+t^j) - B_0^H(y+t^j), 1\leq j\leq p-1 }^{-1/2} \\
    &\leq R^{-1} \cdot \varcb{ B_0^H(x+t^p) - B_0^H(y+t^p) }{ B_0^H(x+t^j) , B_0^H(y+t^j), 1\leq j\leq p-1 }^{-1/2}.
  \end{align*}
  As the reader may expect, we aim to use the anisotropic LND property~\eqref{eq:ani_lnd_fbs} to bound the integral $\int_{\ivff{\eps,1}^N} \Pii(t^p) \,\dt t^p$. To simplify the former expression, we introduce a collection of rectangles $(I_l)_l$ which forms a partition of $\ivff{\eps,1}^N$ and we split the previous integral accordingly.

  More specifically, define for every $k\in\brc{1,\dotsc,N}$,  $\Si_k = \brcb{ t^j_k, t^j_k+x_k-y_k,t^j_k-x_k+y_k ; 1\leq j\leq p-1 }$. Then, for any index $l=(l_1,\dotsc,l_N)\in\brc{1,\dotsc,3(p-1)}^N$, let $I_l$ be the $N$-dimensional rectangle:
  \[
    I_l = \prod_{k=1}^N \ivfobb{ s^{l_k}_k - \frac{s^{l_k}_k-s^{l_k-1}_k}{2}, s^{l_k}_k + \frac{s^{l_k+1}_k-s^{l_k}_k}{2} },
  \]
  where the elements $(s^{l_i}_k; 1\leq l_i\leq 3(p-1))$ of the set $\Si_k$ are assumed to be increasingly sorted. The collection of rectangles $(I_l)_l$ clearly forms a partition of $\ivff{\eps,1}^N$ (choosing accordingly $s^{0}_k$ and $s^{3p-2}_k$ to cover the full square). Consequently, the integration over $t^p$ on the domain $\ivff{\eps,1}^N$ can be reduce to a finite sum of integrals on each element $I_l$.
  Thus, let us now set $l\in\brc{1,\dotsc,3(p-1)}^N$ and observe that for any $t^p\in I_l$,
  \[
    \forall k\in\brc{1,\dotsc,N};\quad \abs{t^p_k - s^l_k} \leq \min_{1\leq j\leq p-1} \abs{ x_k+t^p_k - z^j_k } + \min_{1\leq j\leq p-1} \abs{ y_k+t^p_k - z^j_k },
  \]
  where $z^j$ denotes either $x+t^j$ or $y+t^j$. As a consequence, according to Proposition~\ref{prop:ani_lnd_fbs}, for any $t^p\in I_l$, $\Pii(t^p) \leq c_0 \, R^{-1} \rho(x,y)^{-1} \rho(s^j,t^p)^{-1}$, and thus,
  \begin{align*}
    \int_{I_l} \Pii(t^p)^d \,\dt t^p \leq c_0 R^{-d}  \rho(x,y)^{-d} \int_{I_l} \rho(s^j,t^p)^{-d} \,\dt t^p.
  \end{align*}
  Let us prove the last integral is finite:
  \begin{align*}
    \int_{I_l} \rho(s^j,t^p)^{-d} \,\dt t^p
    \leq c \int_{B(0,1)} \brcbb{ \sum_{\ell=1}^N \abs{u_\ell}^{H_\ell} }^{-d} \dt u
    \leq c \int_{B(0,1)} \norm{v}^{-d} \prod_{\ell=1}^N \abs{v_\ell}^{1/H_\ell-1} \,\dt v,
  \end{align*}
  using the simple change of variable $v_\ell=u_\ell^{H_\ell}$. Then, switching to spherical coordinates,
  \begin{align*}
    \int_{I_l} \rho(s^j,t^p)^{-d} \,\dt t^p
    &\leq c\int_0^1 r^{-d+\sum_{\ell=1}^N 1/H_\ell - 1 } \,\dt r \int_{S^{N-1}} h(\varphi)\,\dt \varphi
    < +\infty,
  \end{align*}
  since $\sum_{\ell=1}^N 1/H_\ell > d$ and the induced function $h$ is bounded on the sphere $S^{N-1}$ ($1/H_\ell-1 > 0$ for every $\ell$). Hence,
  \[
    \int_{\ivff{\eps,1}^N} \Pii(t^p)^d \,\dt t^p \leq c_1 \, p^N R^{-d}  \rho(x,y)^{-d},
  \]
  and by induction on $p$, we obtain Inequality \eqref{eq:lemma_IxyR}.
\end{proof}

\begin{remark}  \label{remark:ineq_ext}
  We may note that the proof of Lemma~\ref{lemma:moments_IxyR} also provides a slighter more general inequality. Namely, for any $\alpha$ such that $d<\alpha<\sum_{\ell=1}^N \frac{1}{H_\ell}$,
  \begin{align*}
    \int_{\ivff{\eps,1}^{Np}} \prb{ \absb{ B_0^H(x+t^j) - B_0^H(y+t^j) } \leq R^{-1} , 1\leq j \leq p }^\alpha \,\dt t^1 \cdots \dt t^p \\
    \leq c_0^p (p!)^N R^{-\alpha p} \rho(x,y)^{-\alpha p},
  \end{align*}
  This extension will be directly used in the proof of Theorem~\ref{th:lebesgue_fbs}.
\end{remark}

The proof of Theorem~\ref{th:weak_unif_dim_fbs} follows the exact same structure as the ones presented by \citet{Khoshnevisan.Wu.ea-2006} and \citet{Wu.Xiao-2007a}. Consequently, we only present the main steps, and refer to the former for the technical details which remain the same.
\begin{proof}[Proof of Theorem \ref{th:weak_unif_dim_fbs}]
  Since $B^H$ is H\"older continuous with respect to the anisotropic metric $\rho$, classic results (see for instance \cite{Xiao-2009a}) on images of fractal sets show that almost surely,
  \[
    \dimH B^H(F+t) \leq \min\brcb{ d, \dimH^\rho F }\quad\text{for all Borel sets $E$ and all }t\in\ivff{0,1}^N.
  \]
  To obtain the lower bound, we first prove that almost surely, there exists $n_0(\omega)$ such that
  \[
    \forall n\geq n_0(\omega),\ \forall x,y\in\ivff{0,1}^N;\quad I(x,y,2^n) \leq c_0\, n^N 2^{-nd} \rho(x,y)^{-d}.
  \]
  The previous property is direct application of Borel--Cantelli lemma and the continuity of fractional Brownian sheets. We refer to \cite{Wu.Xiao-2007a} for the details of the arguments.

  Let us now set $\omega\in\Omega$, a Borel set $F\subset\ivff{0,1}^N$, $\gamma\in\ivoo{0,\dimH F}$ and $\eta\in\ivoo{0,d\wedge\gamma}$. Frostman's lemma implies the existence of a probability measure $\mu$ carrying $F$ and such that
  \[
    \mu(S) \leq c_1\, (\diam^\rho S)^\gamma\quad\text{for any measurable set }S\subset\ivff{0,1}^N.
  \]
  Let $\nu_t$ be the image of $\mu$ by $B^H(\cdot+t)$. Still according to Frostman's lemma, it is sufficient to prove
  \[
    I \eqdef \iint_{R^{2d}} \frac{ \nu_t(\dt u)\,\nu_t(\dt v) }{\norm{u-v}^\eta} < \infty\quad\text{for almost all }t\in\ivff{0,1}^N,
  \]
  to obtain our result.
  Following the idea of \citet{Kaufman-1989}, we have
  \begin{align*}
    I
    &= \iint \frac{ \mu(\dt x)\,\mu(\dt y) }{ \norm{ B^H(x+t) - B^H(y+t) }^\eta } \\
    &= \eta \int_0^\infty \iint \indi_{\ivff{-1,1}}\pthb{ R\,\norm{ B^H(x+t) - B^H(y+t) } } R^{\eta-1} \mu(\dt x)\mu(\dt y) \,\dt R \\
    &\leq 1 + \eta \int_1^\infty \iint \indi_{\ivff{-1,1}}\pthb{ R\,\norm{ B^H(x+t) - B^H(y+t) } } R^{\eta-1} \mu(\dt x)\mu(\dt y) \,\dt R
  \end{align*}
  Integrating the previous integral over $t\in\ivff{0,1}^N$, we thus need to show that
  \[
    J \eqdef \iint \int_1^\infty  I(x,y,R) R^{\eta-1} \dt R\,\mu(\dt x)\mu(\dt y) < \infty.
  \]
  Let $D=\brc{(x,y)\in\ivff{0,1}^{2N} : \rho(x,y) \leq R^{-1}}$ and $J_1$, $J_2$ respectively denote the integral $J$ over the domains $D$ and $D^c$. Since $(\mu\times\mu)(D) \leq c_1 \,R^{-\gamma}$,
  \[
    J_1 \leq c_1 \int_1^\infty R^{-\gamma+\eta-1} \dt R < \infty.
  \]
  Furthermore, as for any $(x,y)$, $I(x,y,R) \leq c_0(\omega)\, R^{-d} \rho(x,y)^{-d}$,
  \begin{align*}
    J_2
    &\leq c_0(\omega) \iint \rho(x,y)^{-d} \,\mu(\dt x)\mu(\dt y) \int_{\rho(x,y)^{-1}}^\infty R^{\eta-d-1} \log(R)^N \,\dt R \\
    &\leq c_2(\omega) \iint \rho(x,y)^{-\eta} \log(\rho(x,y)^{-1})^N \,\mu(\dt x)\mu(\dt y) < \infty.
  \end{align*}
  The last two inequalities complete the proof of Theorem \ref{th:weak_unif_dim_fbs}.
\end{proof}

The second part of this section is devoted to the proof of Theorem~\ref{th:lebesgue_fbs}. As previously, the sketch of the latter is highly inspired by the original work of \citet{Kaufman-1989}, and we therefore focus on differences compared to the previous results presented by \citet{Khoshnevisan.Wu.ea-2006} and \citet{Wu.Xiao-2007a}.
\begin{proof}[Proof of Theorem \ref{th:lebesgue_fbs}]
  Since $\dimH^\rho(F)>d$, there exists a probability measure $\mu$ on $F$ such that $\iint_{R^{2N}} \frac{\mu(\dt s)\,\mu(\dt t)}{\rho(s,t)^d} < \infty$.
  To prove that $\lambda_d\pthb{B^H(F+t)} > 0$, it is sufficient to show that
  \[
    \text{a.s. }\int_{\ivff{0,1}^N} \int_{\R^d} \abs{\widehat{\nu}_t(u)}^2 \,\dt u \,\dt t < \infty.
  \]
  where $\widehat{\nu}_t(u) = \int_{\R^N} e^{i\scpr{u}{B^H(x+t)}} \,\mu(\dt x)$ and the exceptional set does not depend on $t$.

  Let $\psi\geq 0$ be a smooth function on $\R^d$ such that $\psi(u)=1$ when $1\leq\abs{u}\leq 2$ and $\psi(u)=0$ outside $1/2\leq\abs{u}\leq 5/2$. Since $\int_{\abs{u}>1} \abs{\widehat{\nu}_t(u)}^2 \,\dt u$ is bounded above by
  \[
    \sum_{n=0}^\infty 2^n \iint_{R^{2N}} \widehat{\psi}\pthb{ 2^n B^H(x+t) - 2^n B^H(y+t) } \,\mu(\dt x)\mu(\dt y),
  \]
  it remains to prove
  \[
    \sum_{n=0}^\infty 2^n \int_{\ivff{0,1}^N} \iint_{R^{2N}} \widehat{\psi}\pthb{ 2^n B^H(x+t) - 2^n B^H(y+t) } \,\mu(\dt x)\mu(\dt y) \,\dt t < \infty.
  \]
  For that purpose, we study in the following lemma the component $J$ defined by
  \[
    J(x,y,n) \eqdef \int_{\ivff{0,1}^N} \widehat{\psi}\pthb{ 2^n B^H(x+t) - 2^n B^H(y+t) } \,\dt t.
  \]
\end{proof}

\begin{lemma}  \label{lemma:ineq_J}
  There exist two positive constants $c_{0}$ and $\beta$ such that with probability $1$, for all $n\geq n(\omega)$ and $\rho(x,y)\geq c_{0} n^N 2^{-n}$,
  \[
    J(x,y,n) \leq (2+\beta)^{-n} \rho(x,y)^{-d}.
  \]
\end{lemma}
\begin{proof}
  As observed originally by \citet{Kaufman-1989}, using the reasoning presented in the proof of Theorem~\ref{th:weak_unif_dim_fbs} and a Borel--Cantelli argument, it is sufficient to show the existence of positive constants $\beta$, $c_{3,1}$ and $c_{3,2}$ independent of $p$ such that
  \begin{align}  \label{eq:bound_J}
    \esp{ J(x,y,n)^{2p} } \leq c_{3,1}^{p} n^{c_{3,2}p} (2+\beta)^{-2np} \rho(x,y)^{-2pd}.
  \end{align}
  Namely, we need to upper bound the following term
  \begin{align*}
    \espbb{ \int_{\ivff{\eps,1}^{2Np}} \prod_{j=1}^{2p} \absb{ \widehat\psi(2^n B^H(x+t^j)) - \widehat\psi(2^n B^H(y+t^j)) } \dt \mathbf{t} },
  \end{align*}
  where $\mathbf{t} \eqdef (t^1,\dotsc, t^{2p})$, $t^j\in\R^N$.
  For any $n\in\N$, let $S_n$ be the following set
  \begin{align*}
    S_n = \bigcup_{k=1}^{2p} \bigcup_{\ell=1}^N \brcB{ \mathbf{t}\in\ivff{\eps,1}^{2Np} : \abs{t^k_\ell-t^j_\ell} > r_n^{1/H_\ell} \text{ and } \abs{x_\ell+t^k_\ell-t^j_\ell-y_\ell} > r_n^{1/H_\ell} \quad \forall j\neq k } ,
  \end{align*}
  where $r_n\eqdef c_{3,0}\,(n+1)2^{-n}$. We will begin by studying the former integral on the domain $S_n$. It takes the following equivalent form:
  \begin{align*}
    &\Esp \int_{S_n} \int_{\R^{2pd}} \prod_{j=1}^{2p} \exp\brcb{ i\scpr{ \xi^j }{ 2^n B^H(x+t^j) - 2^n B^H(y+t^j) } } \,\dt\xi\,\dt \mathbf{t} \\
    &= \int_{S_n} \int_{\R^{2pd}} \exp\brcbb{ -\frac{1}{2} \sum_{\ell=1}^d \varbb{ \sum_{j=1}^{2p} \xi^j_\ell \bktb{ B^H(x+t^j) - 2^n B^H(y+t^j) } }  } \prod_{j=1}^{2p} \psi(\xi^j) \,\dt\xi\,\dt \mathbf{t}.
  \end{align*}
  Since $1/2\leq \abs{\xi^k}\leq 5/2$, there exists $\ell_0\in\brc{1,\dotsc,d}$ such that $\xi^k_{\ell_0} \geq (2\sqrt{d})^{-1}$. Hence, owing the classic sectorial LND property \eqref{eq:sec_local_nondet1},
  \begin{align*}
    \varbb{ \sum_{j=1}^{2p} \xi^j_\ell \bktb{ B^H(x+t^j) - 2^n B^H(y+t^j) } }
    &\geq c_{0}\, 2^{2n} \sum_{\ell=1}^N \min_{j\neq k} \brcB{ \abs{t^k_\ell - t^j_\ell}^{2H_\ell}, \abs{x_\ell+t^k_\ell - t^j_\ell-y_\ell}^{2H_\ell} } \\
    &\geq c_{1} \, (n+1)^2.
  \end{align*}
  Hence,
  \[
    \Esp \int_{S_n} \int_{\R^{2pd}} \prod_{j=1}^{2p} \exp\brcb{ i\scpr{ \xi^j }{ 2^n B^H(x+t^j) - 2^n B^H(y+t^j) } } \,\dt\xi\,\dt \mathbf{t} \leq e^{-c_{2} n^2},
  \]
  where $c_{2}$ can be chosen sufficiently large up to a modification of $c_{3,0}$. As a consequence, the previous term will clearly be negligible compared to the bound we aim to obtain.

  Let us now consider the second integral over the domain $T_n\eqdef\ivff{0,1}^{2Np} \setminus S_n$, and first note that $T_n$ can be written as
  \begin{align*}
    T_n = \bigcap_{k=1}^{2p} \bigcap_{\ell=1}^N \pthbb{
    &\brcB{ \mathbf{t}\in\ivff{\eps,1}^{2Np} : \min_{j_{\ell,1}\neq k} \absb{t^k_\ell-t^{j_{\ell,1}}_\ell} \leq r_n^{1/H_\ell} }  \\
    \cup&\brcB{ \mathbf{t}\in\ivff{\eps,1}^{2Np} : \min_{j_{\ell,2}\neq k} \absb{x_\ell+t^k_\ell-t^{j_{\ell,2}}_\ell-y_\ell} \leq r_n^{1/H_\ell} } }.
  \end{align*}
  We easily observe that $T_n$ is the union of at most $(4n)^{2Np}$ sets of the following form:
  \begin{align*}
    A_{\mathbf{j}} = \brcb{ \mathbf{t}\in\ivff{\eps,1}^{2Np} :  \absb{z_\ell+t^k_\ell-t^{j_{\ell,k}}_\ell} \leq r_n^{1/H_\ell}, \forall k\in\brc{1,\dotsc,2p}, \forall\ell\in\brc{1,\dotsc,N}},
  \end{align*}
  where $z_\ell=0$ or $x_\ell-y_\ell$. Up to a permutation of indices, the previous set can be written as
  \begin{align*}
    A_{\mathbf{j}} = \bigtimes_{\ell=1}^N \brcb{ \mathbf{t}\in\ivff{\eps,1}^{2p} :  \absb{z_\ell+t^k_\ell-t^{j_{\ell,k}}_\ell} \leq r_n^{1/H_\ell}, \forall k\in\brc{1,\dotsc,2p}}.
  \end{align*}
  As a consequence, using Lemma~3.8 proved in \cite{Khoshnevisan.Wu.ea-2006}, the Lebesgue measure of the previous set can be bounded as following:
  \begin{align*}
    \lambda_{2Np}\pth{A_{\mathbf{j}}} \leq
    2^{Np} \,r_n^{p\sum_{\ell=1}^N \tfrac{1}{H_\ell}},
  \end{align*}
  providing a bound on the measure of the full set $T_n$: $\lambda_{2Np}\pth{T_n} \leq c_4^{p}\,n^{2Np} \, r_n^{p\sum_{\ell=1}^N \tfrac{1}{H_\ell}}$.
  Then, we divide the integral
  \begin{align*}
    &\int_{T_n} \espbb{ \prod_{j=1}^{2p} \absB{ \widehat\psi(2^n B^H(x+t^j)) - \widehat\psi(2^n B^H(y+t^j)) } } \dt \mathbf{t} \\
  \end{align*}
  into two parts $I_1$ and $I_2$, respectively conditioning the former with respect to the events $D_n$ and $D_n^c$, where
  \begin{align*}
    D_n = \brcB{ \max_{1\leq j\leq 2p} \normb{ B^H(t+x^j) - B^H(y+x^j) } > 2^{-(1-\eps)n} }.
  \end{align*}
  Since $\widehat\psi$ is rapidly decreasing, there exists a constant $c_5>0$ which can be chosen as large as possible such that
  \begin{align*}
    I_1 \leq c_6^{p}n^{c_7 p} \,2^{-np\sum_{\ell=1}^N \tfrac{1}{H_\ell}} \exp\pth{-c_5\, n}.
  \end{align*}
  Recall that we may assumed that $N\geq 2$, inducing that $\sum_{\ell=1}^N \tfrac{1}{H_\ell} > 2$, meaning that the previous bound is negligible compared to the right end term in Equation~\eqref{eq:bound_J} ($\beta$ can be chosen small enough).

  Finally, we may conclude the proof by bounding the term $I_2$. Let us set $\alpha$ such that $d<\alpha<\sum_{\ell=1}^N \tfrac{1}{H_\ell}$ and observe
  \begin{align*}
    &\int_{T_n} \prB{ \max_{1\leq j\leq 2p} \normb{ B^H(x+t^j)) - B^H(y+t^j)) } \leq 2^{-(1-\eps)n} } \dt \mathbf{t} \\
    &=\int_{T_n} \prB{ \max_{1\leq j\leq 2p} \absb{ B_0^H(x+t^j)) - B_0^H(y+t^j)) } \leq 2^{-(1-\eps)n} }^d \dt \mathbf{t} \\
    &\leq \lambda_{2Np}\pth{T_n}^{1-\tfrac{d}{\alpha}} \bktbb{ \int_{\ivff{\eps,1}^{2Np}} \prB{ \max_{1\leq j\leq 2p} \absb{ B_0^H(x+t^j)) - B_0^H(y+t^j)) } \leq 2^{-(1-\eps)n} }^\alpha \dt \mathbf{t} }^{\tfrac{d}{\alpha}}
  \end{align*}
  using the classic Hölder inequality. Based on Remark~\ref{remark:ineq_ext} and the previous estimates, we get
  \begin{align*}
    \int_{T_n} \prB{ \max_{1\leq j\leq 2n} \normb{ B^H(x+t^j)) - B^H(y+t^j)) } \leq 2^{-(1-\eps)n} } \dt \mathbf{t} \\
    \leq c_8^p \, p^{c_9 n} 2^{-np\pthb{ 2d(1-\eps) + \pthb{1-\tfrac{d}{\alpha}}\sum_{\ell=1}^N \tfrac{1}{H_\ell} }} \rho(x,y)^{-2pd}
  \end{align*}
  Then, we may observe that $\eps$ can be chosen as small as wanted, and particularly, such that $2d(1-\eps) + \pthb{1-\tfrac{d}{\alpha}}\sum_{\ell=1}^N \tfrac{1}{H_\ell} > 2$.

  The combination of the three previous bounds clearly shows the existence of $\beta>0$ such that
  \begin{align*}
     \esp{ J(x,y,n)^{2p} } \leq c_{3,1}^{p} p^{c_{3,2}n} (2+\beta)^{-2np} \rho(x,y)^{-2pd},
  \end{align*}
  where the constants $c_{3,1}$ and $c_{3,2}$ are independent of $p\in\N$ and $n\in\N$.
\end{proof}

\begin{proof}[Proof of Theorem \ref{th:lebesgue_fbs}]
  To conclude the proof of Theorem~\ref{th:lebesgue_fbs}, we simply observe that the previous Lemma entails
  \begin{align*}
    \sum_{n=0}^\infty 2^n \iint_{R^{2N}} \abs{J(x,y,n)} \,\mu(\dt x)\mu(\dt y)
    \leq \sum_{n=0}^\infty 2^n (2+\beta)^{-n} \iint_{R^{2N}} \rho(x,y)^{-d} \,\mu(\dt x)\mu(\dt y)
    < \infty.
  \end{align*}
\end{proof}


\end{document}